\documentclass[11pt,leqno]{amsproc}
\usepackage[a4paper,top=3cm,bottom=3cm,left=3cm,right=3cm]{geometry}
\usepackage{amssymb,amsmath}
\usepackage{exscale}
\usepackage{tikz}

\usepackage{color} 

\RequirePackage[cp1251]{inputenc}
\RequirePackage{srcltx}

\numberwithin{equation}{section}
\renewcommand\r{\rangle}
\renewcommand\l{\langle}

\newcommand\dist{\operatorname{dist}}

\newcommand\cal{\mathcal}

\newcommand\R{\mathbb{R}}
\newcommand \C{\mathbb{C}}
\newcommand\Q{\mathbb{Q}}
\newcommand\Z{\mathbb{Z}}
\newcommand\N{\mathbb{N}}

 \newcommand\norm[1]{\left\lVert#1\right\rVert}

\newcommand\B{\mathcal{B}}
\newcommand{\E} {{\mathcal E}}

\newtheorem{Thm}{Theorem}[section]
\newtheorem{Lemma}[Thm]{Lemma}

\theoremstyle{remark}
\newtheorem{Rem}[Thm]{Remark}

\begin{document}

     \title{Examples of  exponential bases  on unions of intervals}
\author{Oleg Asipchuk}
\address{Oleg Asipchuk,  Florida International University,
	Department of Mathematics and Statistics,
	Miami, FL 33199, USA}
\email{aasip001@fiu.edu}
\author{Vladyslav Drezels}
\address{ Vladyslav Drezels, Florida International University,
	Department of Mathematics and Statistics,
	Miami, FL 33199, USA}
\email{vdrez001@fiu.edu}
\begin{abstract}
 In this paper, we construct explicit exponential bases  of unions of segments of total measure one. Our construction  applies to finite or infinite unions of segments,  with some conditions on the gaps between them. 
   	We also  construct exponential bases on finite or infinite unions of cubes in $\R^d$ and prove a stability result for unions of segments that generalize    Kadec's $\frac 14-$theorem.
\end{abstract}
\maketitle
\section{Introduction}
The main purpose of this paper is to construct explicit exponential bases on finite or infinite unions of intervals of the real line. We assume that our intervals have total measure  one, until otherwise specified. 

We recall that an exponential basis on a domain $D\subset \R^d$ is an unconditional  Sch\"auder basis for $L^2(D)$ in the form of  $\{e^{2\pi i \lambda_n \cdot x}\}_{n\in\Z^d  }$ with  $\lambda_n \in\R^d$.   An important example of exponential basis is the 
%set $\E=\{e^{2\pi i n \cdot x}\}_{n\in\Z^d  }$, with $n=(n_1, ...,\, n_d)\in \Z^d$. $\E$ is an orthonormal basis of the unit square $Q=[0,1]^d$, a multidimensional version of the  
Fourier basis $ \E=\{e^{2\pi i n   x}\}_{n\in\Z   }$  for $L^2(0,1)$.

Non-orthonormal exponential bases on  intervals of the real line  are well studied and well understood  in the context of non-harmonic Fourier
series, see  \cite{L, NL, P, S} just to cite a few.  Proving the existence of an exponential basis is in general a difficult problem and   constructing explicit bases can be even more difficult.

 In \cite{NL} the Authors proved the existence of bases on  finite unions of intervals under some conditions on the lengths of the intervals. It is proved in \cite{KN} that  exponential bases on any finite union of intervals exist, but the construction of such bases is not explicit. It is not clear whether exponential bases  on arbitrary infinite unions of intervals  exist or not.

In \cite{D}   the Author proves necessary and sufficient conditions  for which  sets  in the form of   $\{e^{2\pi i (n+\delta_{ j})x}\}_{n\in\Z,{j\leq N}}$  are exponential bases on unions of $N$  intervals of a unit length separated by integer gaps and gave an explicit expression for the frame constants of these bases.  Such a result  can be used to construct explicit exponential bases on intervals with rational endpoints. However, the conditions involve evaluating the eigenvalues of $N\times N$ matrices,  which can  be a difficult task. 
Some of the results in \cite{D} appear also in other papers, for example in \cite{KMN}. 

In \cite{La}  it is proved that  if $D$ is the union of two disjoint intervals  of total length 1, then $D$  has an orthonormal basis of exponentials if and only if it tiles $\R$ by translations. Recall that a measurable set  $D $  tiles  $\R^d$   by translation if we can fill the space with translated copies of  $D$ without overlaps. Results for  unions of three intervals that tile the real line are in \cite{DKKM}.

In many applications, such as aircraft instrument communications, air
traffic control simulation, or telemetry \cite{F}, one can consider the possibility
of obtaining sampling expansion which involved sample values of a function
and its derivatives.  That translated into finding bases  of the form of $\{x^k e^{2\pi i x\lambda_n}\}_{n\in\Z}$, with $k \in\N\cup\{0\}$. See \cite{GS}.

It is worth mentioning the recent \cite{PRW}, where the Authors partition the interval $[0,1]$ into intervals $I_1$, ... $I_n$ and the set $\Z$ into $\Lambda_1$, ... $\Lambda_n$ such that the complex exponential functions with frequencies in $\Lambda_k$ form a Riesz basis for $L^2(I_k)$.

The existence of orthonormal bases on a domain of $\R^d$ is a difficult  problem related to the tiling properties of the domain.  It has been recently proved in \cite{LM} that convex sets tile $\R^d$ by translations  if and only if  they have  an exponential basis. In \cite{DMT} it is proved that the set $\E =\{e^{2\pi i n\cdot x}\}_{n\in\Z^d}$ is an exponential basis on a domain $D\subset \R^d$ of measure $1$ if and only if $D$ tiles $\R^d$. Furthermore, $\E$ is orthonormal for $L^2(D)$.

The aforementioned  results   in \cite {DMT} are related   to Theorem 1 in \cite{IK}, where it is proved that if a set $\{e^{2\pi i x \lambda}\}_{\lambda\in\Lambda}$ is an orthonormal basis on a domain  $D\subset \R$, then  $\Lambda$ is periodic, i.e.,  $\Lambda= T+\Lambda$ for some $T\in\N$.

\subsection{Our results}
Before we introduce our results we need some more notations. By $m\in\N\cup\{\infty\}$ we mean  that $m$ is either a natural number or it is  infinity. 	We let   $I=[0,1)$.  Given  $0=a_0<a_1<...<a_{m-1}<...<1$ be a sequence for $m\in\N\cup\{\infty\}$, we let $I_j=[a_j, a_{j+1})$, so that $I=  \bigcup I_j.$  
	
	Given $0=b_0<b_1<...<b_{m-1}<...$, we let 
\begin{equation}\label{defJ} J=\bigcup_j J_j, \quad\mbox{with }\ 
 J_j=[a_{j}+b_j,a_{j+1}+b_j).\end{equation}  
We can also write \begin{equation}\label{defJ2} J=\bigcup_j [c_j-\gamma_j, c_j+\gamma_j) \quad\mbox{with }\ 
 c_j=\frac{a_{j-1}+a_j}2+ b_j, \quad \gamma_j= \frac{a_j-a_{j-1}} 2.\end{equation}
   Note that $\sum_j \gamma_j=1$. So, on graph we have

   	  
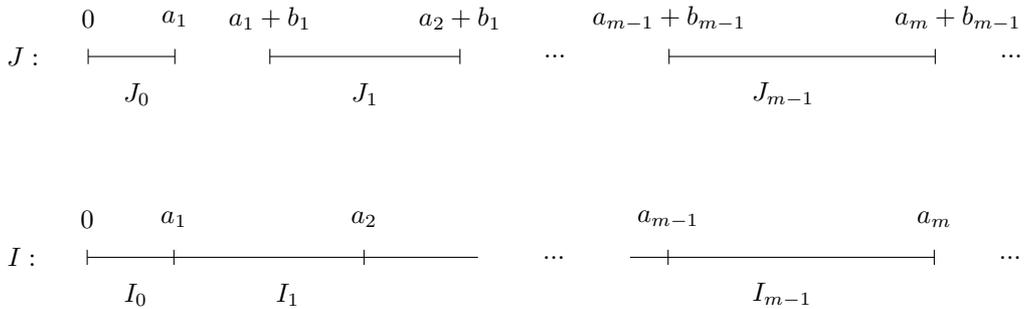
\begin{figure}[h!]
  \begin{center}
     \begin{tikzpicture}
      \draw [black] (-3.14,0) --(-2,0);
      \draw [black] (-3.14,-0.1) --(-3.14,0.1);
      \draw [black] (-2,-0.1) --(-2,0.1);
      \draw [black] (-0.75,0) --(1.75,0);
      \draw [black] (-0.75,-0.1) --(-0.75,0.1);
      \draw [black] (1.75,-0.1) --(1.75,0.1);
      \draw [black] (4.5,0) --(8,0);
    \draw [black] (4.5,-0.1) --(4.5,0.1);
      \draw [black] (8,-0.1) --(8,0.1);
        \draw    (-3.14,0.5) node [black]{{\small $0$}};
        \draw    (-2,0.5) node [black]{{\small $a_1$}};
        \draw    (-2.5,-0.5) node [black]{{\small $J_0$}};
        \draw    (-0.75,0.5) node [black]{{\small $a_1+b_1$}};
        \draw    (1.75,0.5) node [black]{{\small $a_2+b_1$}};  
        \draw    (0.5,-0.5) node [black]{{\small $J_1$}};
        \draw    (3,0) node [black]{{\small ...}};
        \draw    (4.5,0.5) node [black]{{\small $a_{m-1}+b_{m-1}$}};
        \draw    (8.3,0.5) node [black]{{\small $a_m+b_{m-1}$}};  
         \draw    (6,-0.5) node [black]{{\small $J_{m-1}$}};
                  \draw    (-4,0) node [black]{{\small $J:$}};
                  \draw    (9,0) node [black]{{\small ...}};
      \end{tikzpicture}
   \end{center}
 \end{figure}

	  \begin{figure}[h!]
  \begin{center}
     \begin{tikzpicture}
      \draw [black] (-3.14,0) --(2,0);
      \draw [black] (-3.14,-0.1) --(-3.14,0.1);
      \draw [black] (-2,-0.1) --(-2,0.1);
      \draw [black] (0.5,-0.1) --(0.5,0.1);
      \draw [black] (4,0) --(8,0);
    \draw [black] (4.5,-0.1) --(4.5,0.1);
      \draw [black] (8,-0.1) --(8,0.1);
        \draw    (-3.14,0.5) node [black]{{\small $0$}};
        \draw    (-2.5,-0.5) node [black]{{\small $I_0$}};
        \draw    (-2,0.5) node [black]{{\small $a_1$}};
        \draw    (0.5,0.5) node [black]{{\small $a_2$}};  
        \draw    (-0.5,-0.5) node [black]{{\small $I_1$}};
        \draw    (3,0) node [black]{{\small ...}};
        \draw    (4.5,0.5) node [black]{{\small $a_{m-1}$}};
        \draw    (8,0.5) node [black]{{\small $a_m$}};  
         \draw    (6,-0.5) node [black]{{\small $I_{m-1}$}};
                  \draw    (-4,0) node [black]{{\small $I:$}};
                  \draw    (9,0) node [black]{{\small ...}};
      \end{tikzpicture}
   \end{center}
 \end{figure}

Let $\{x\}$ be the  decimal part of a real number $x,$ let $\lfloor x\rfloor$ (the {\it floor function}) be the largest integer  that is $\leq x$, and let $\lceil x\rceil$ (the {\it ceiling function}) be the smallest integer  that is $\geq x$.   If $E\subset \R$,  we define the distance function for $x\in\R$ and $E\subset \R$ as $\dist(E,x):=\min_{y\in E}\{ |y-x|\}$ or  $\dist(E,F):=\min_{y\in E, x\in F}\{ |y-x|\}.$ 

Next, we introduce the sequence $\delta_n^*.$ Let $\beta>0$ be fixed;  for every $n\in \Z$, we let
	\begin{equation} \label{d*}
		\delta_n^*=\begin{cases}
	\frac{\lfloor \beta n \rfloor+1}{\beta}-n  & \text{if } \{\beta n\} \geq \frac12 \\
	\frac{\lfloor \beta n \rfloor}{\beta}-n  & \text{if } \{\beta n\} < \frac12\\
\end{cases}.
	\end{equation}
We can see at once that 
$$\delta_n^*= \frac{\xi_n}{\beta} \min\left\{\{\beta n\},1-\{\beta n\}\right\}=\frac{\xi_n}{\beta}\dist{(\Z,\beta n)},
$$
where
\begin{equation*}
    \xi_n = \begin{cases}
  1  & \text{if } \{\beta n\} \geq \frac12 \\
  -1  & \text{if } \{\beta n\} < \frac12 \\
\end{cases}
\end{equation*}
Let \begin{equation}\label{def-*}
	 \B^* =\{e^{2\pi i x(n+\delta_n^*)}\}_{n\in\Z}, \end{equation}   where the $\delta_n^*$ are defined as in \eqref{d*}. Our main results are the following.

 \begin{Thm}\label{l-segments}  Let $J$ be as in \eqref{defJ}. If the are exists $\beta\geq 1$ such that  $\frac{b_k}{\beta}\in\Z $ for all $k$, then the set $\B^*$ defined in \eqref{def-*} is an exponential basis for $L^2(J).$ 

\end{Thm}

\begin{Thm}\label{T-ComplementB}
Let  $J$ and $\beta$ be as in Theorem \ref{l-segments}. Let $m\in\N$ and  
%$\beta$ satisfies the conditions of Theorem \ref{l-segments}. Also, 
let $\Lambda = \left\{n+\delta^*_n\right\},$ where $\delta^*$ is as in $\eqref{d*}.$
If  $\frac{b_k}{\beta}\in\Z $ for all $k=1,...,m-1$, then the set $\E(\frac{1}{\Delta}\Z\setminus\Lambda)$ is an exponential basis for $L^2([0,\Delta]\setminus J),$ where 
$$
\Delta=\left\lceil\frac{1+b_{m-1}}{\beta} \right\rceil \beta.
$$
\end{Thm}

    We will prove in Lemma $\ref{lmDelta*}$  that $\B^*$ is also a basis in $L^2(0,1);$  our proof of Theorem $\ref{l-segments}$ shows that $\B^*$ has the same  frame constants in $L^2(J)$ and $L^2(0,1).$

	   In our Theorem, $\ref{l-segments}$ the gaps $ b_k$ are integer multiples of $\beta$ and so  the set  $J$ is unbounded when $m=\infty$.  To the best of our knowledge, there are very few examples of exponential bases on unbounded sets in the literature. 
	   %The existence of exponential bases on  infinite  unions of intervals has not been proven in general. 
	    The existence of exponential frames on unbounded sets  of finite measure has been recently proved in \cite{NOU}.
		\medskip
		
We also observe that in  Theorem $\ref{T-ComplementB} $ we can't consider $m=\infty,$ because  $J$ is unbounded, and in the proof, we need to consider a finite  
interval $[0,\Delta]$ that contains  $J$.

Our paper is organized as follows: In section 2 we recall some preliminaries and we prove some important lemmas. In Section 3 we prove our main results. In Section 4 we prove the result for unions of cubes in $\R^d$ and a stability result.   

\medskip
\noindent
{\it Acknowledgments}. Foremost, we would like to express our sincere gratitude to our mentor Professor Laura De Carli for her continuous support of our research. 

\medskip
The research of this project started during the Summer 2022 REU program "AMRPU @ FIU" that took place at the Department of Mathematics and Statistics, Florida International University, which was supported by the NSA grant H982302210016 and NSF (REU Site) grant DMS-2050971. We acknowledge the collaboration of Mark Leal and Izzy Morell in the initial stage of this project and their participation in the technical report (available at go.fiu.edu/amrpu).

\section{Preliminaries}
We have used the excellent textbooks \cite{Heil} and  \cite{Y} for the definitions and some of the results presented in this section.

Let  $H$   be a  separable Hilbert space with inner product $\langle\ ,\ \rangle $  and norm $||\ ||=\sqrt{\l \ , \    \r} $. We will mostly work with $L^2(D),$ where $D\subset\R^d.$ So, the norm will be $||f||_2^2=\int_D |f(x)|^2 dx.$ The characteristic function on $D$ we denote as $\chi_D.$   
  
A sequence of vectors ${\mathcal V}= \{v_j\}_{j\in\Z }$  in   $H$ is a
{\it Riesz basis} if 
there exist  constants $A, \ B>0$    such that, for any $w\in H$ and  for all finite sequences $ \{a_j\}_{j\in J}\subset\C $, the following inequalities hold. 
\begin{align}\label{e2-RS}
 	A  \sum_{j\in J}   |a_j|^2   \leq  &\Big\Vert \sum_{j\in J}  a_j  v_j \Big\Vert^2  \leq B \sum_{j\in J} |a_j|^2.
\\\label{e2-frame}
 A||w||^2\leq  &\sum_{j=1}^\infty |\l  w, v_j\r |^2\leq B ||w||^2.
 \end{align}
The constants $A$ and $B$ are called {\it frame constants} of the basis. The left inequality in \eqref{e2-RS}  implies that ${\cal V}$ is linearly independent,  and the left inequality in \eqref{e2-frame} implies that  ${\cal V}$ is complete. If condition \eqref{e2-RS} holds we call ${\mathcal V}$ a Riesz sequence. We call ${\mathcal V}$ a frame if the condition \eqref{e2-frame} holds. If the condition \eqref{e2-frame} holds and $A=B$ then we call  ${\mathcal V}$ a tight frame, If $A=B=1,$ then we have a Parseval frame.
The following lemma is well-known, but for the reader's convenience, we will prove it.
\begin{Lemma}\label{lmBconstant}
If a sequence of vectors ${\mathcal V}= \{v_j\}_{j\in\Z }$ is a frame with upper constant $B$, then the right inequality in $\eqref{e2-RS}$ holds, i.e. for all finite sequences $ \{a_j\}_{j\in J}\subset\C $ 
$$
\Big\Vert \sum_{j\in J}  a_j  v_j \Big\Vert^2  \leq B \sum_{j\in J} |a_j|^2
$$
\end{Lemma}
\begin{proof}
${\mathcal V}$ is a frame, so for all finite sequences $ \{a_j\}_{j\in J}\subset\C $ there is $f\in H$ such that $f=\sum_{j\in J}  a_j  v_j.$ So,
\begin{equation*}
    \begin{split}
        ||f||^2&=\l\sum_{j\in J}  a_j  v_j,f\r=\sum_{j\in J}  a_j\l v_j,f\r\\
        &\leq \sqrt{\sum_{j\in J}  |a_j|^2} \sqrt{\sum_{j\in J}  |\l v_j,f\r|^2}\leq \sqrt{\sum_{j\in J}  |a_j|^2} \sqrt{B}||f||.
    \end{split}
\end{equation*}
Therefore, 
$$
\Big\Vert \sum_{j\in J}  a_j  v_j \Big\Vert  \leq \sqrt{\sum_{j\in J}  |a_j|^2} \sqrt{B}.
$$
\end{proof}

One more lemma that makes a connection between frames and bases.
\begin{Lemma}
If  $E(\Lambda)$ is  basis for $L^2(D)$ and  $D'\subset D$, then $E(\Lambda)$ is a frame for $L^2(D')$  with at least the same frame constants. In particular, if $E(\Lambda)$ is an orthogonal basis for $L^2(D)$, then it is a tight frame for $L^2(D')$
\end{Lemma}
\begin{proof}
    Let $E(\Lambda)$ is a basis for $L^2(D),$ then for all $w\in L^2(D)$
    $$
    A||w||^2\leq  \sum_{j=1}^\infty |\l  w, v_j\r |^2\leq B ||w||^2.
    $$
    Also, for any $f\in L^2(D'),$ there is $w\in L^2(D)$ such that $f=w\chi_{D'}.$ So, the frame inequalities hold for any $f\in L^2(D').$ Therefore, $E(\Lambda)$ is a frame for $L^2(D')$  with at least the same frame constants.
\end{proof}
\medskip
An  important characterization of Riesz bases is that they are bounded and unconditional  Sch\"auder bases. See e.g \cite{Heil}.

Let $\vec v \in\R^d$ and $\rho>0$; we  denote with  $d_\rho D =\{ \rho x \ : \:   x\in D\}$ and by $ t_{\vec v} D = \{x+\vec v \ : \  x\in D\}$ the dilation and translation of $D$. Sometimes we will write $\vec v+D$  instead of $t_{\vec v} D$ when there is no risk of confusion. 

The following lemma can easily be proved with a change of variables in \eqref{e2-RS} and \eqref{e2-frame}. 
\begin{Lemma}\label{L-dil-basis}
 Let $\vec v\in\R^d$ and $\rho>0$. 
 The set ${\cal V}=\{ e^{2\pi i \l x, \lambda_n\r }\}_{n\in\Z}$ is a Riesz basis for $L^2(D)$ with constants $A$ and $B$ if and only if  the set  $\{ e^{2\pi i \l x, \, \frac 1\rho \lambda_n \r }\}_{n\in\Z} $ is a Riesz basis for  $L^2(t_{\vec v}(d_\rho D ))$ with constants $A \rho^{ d}$ and $B \rho^{ d}$.
\end{Lemma}

 \subsection{Paley-Wiener and Kadec stability theorem}

Bases in Banach spaces are stable, in the sense that  small perturbations of a basis still produce bases.

One of the fundamental stability criteria, and historically the first, is due to R. Paley and N. Wiener in \cite {PW}. 

\begin{Thm}[Paley--Wiener Theorem]\label{T-PaleyWiener}
	Let $\{x_n\}_{n\in\N}$ and $\{y_n\}_{n\in\N}$ be sequences in a Banach space $X$. Let  $\lambda$ be a real number $(0<\lambda<1)$ such that 
	\begin{equation*}
		\norm{\sum_n a_n(x_n-y_n)}\leq \lambda \norm{\sum_n a_n x_n} 
	\end{equation*}
	holds for any arbitrary finite set of scalars $\{a_n\}\subset \C.$ Then if $\{x_n\}$ is a basis so is $\{y_n\}.$ Moreover, if $\{x_n\}$ has Riesz constants $A$ and $B,$ then 
	%the following estimation is true 
	$$
	(1-\lambda) A  \sum   |a_n|^2   \leq  \left\Vert \sum  a_n  y_n \right\Vert^2  \leq (1+\lambda) B \sum |a_n|^2. 
	$$
\end{Thm}

\noindent
We will use the following important observation: If $\B=\{x_n\}_{n\in\N}$ is a bounded and unconditional  basis in a Banach space $(X, \ ||\ ||) $, and  $||\ ||_*$ is a norm  equivalent to $||\ ||$,  then $\B$ is also a bounded and unconditional  basis in $(X,\ || \ ||_*)$. Note that two norm $||\ ||_*$ and $||\ ||$ are equivalent if there are two constants $c$ and $C$ such that for all elements of the space $c||x||_*\leq||x||\leq C||x||_*.$  

Let  a set $\{f_n\}_{n\in\N}$ be a Riesz basis for $L^2(D)$ with norm $ ||\  ||_2$  if $||\ ||_*$  is  equivalent  to $||\ ||_2$, then $\{f_n\}_{n\in\N}$ is a bounded and unconditional basis of $(L^2(D), \ ||\ ||_*).$ So, if a sequence $\{g_n\}_{n\in\N}\subset  L^2(D)$ satisfies the conditions of the Paley--Wiener theorem with respect to the norm  $||\ ||_*$, i.e. 
$$
\norm{\sum_n c_n(f_n-g_n)}_*\leq \lambda \norm{\sum_n c_n f_n}_*
$$
with $0<\lambda<1,$
then $\{g_n\}_{n\in\N}$ is a bounded and unconditional basis of $(L^2(D), \ ||\ ||_*)$ and hence also of $(L^2(D), \ ||\ || )$. Thus, $\{g_n\}_{n\in\N}$ is a Riesz basis in $(L^2(D), \ ||\ || )$. This observation proves the following Lemma, which will be useful later on.

\begin{Lemma}
	Let $m\in\N$ and $D=\cup_{j=1}^m D_j\subset \R^d$ where $D_j\cap D_k=\emptyset.$  Let $\{g_n\}_{n\in\ N}$ be a Riesz basis for $L^2(D)$. Let $ \{h_n\}_{n\in \Z}  \subset L^2(D)$ be such that for every finite sequence  $\{a_n\}\in \C$
	$$
  \sup_{j\leq m}	\norm{\sum_n a_n(g_n-h_n)}_{L^2(D_j)}\leq \alpha \sup_j	\norm{\sum_n a_n g_n}_{L^2(D_j)}.
$$
Then, the set $ \{h_n\}_{n\in \Z}$ is a Riesz basis for $L^2(D)$.

\end{Lemma}

\medskip
The celebrated Kadec stability theorem (also called  Kadec $\frac 14$ theorem)  gives an optimal measure of how the standard  orthonormal basis  $\E=\{e^{2\pi i n x}\}_{n\in\Z }$ on  the unit interval $[0,1]$ can be perturbed to  still obtain an exponential basis.

\begin{Thm} \label{T-Kodec}
	Let $\Lambda =  \{\lambda_n \}_{n\in\Z}  $ be a sequence in $\R$ for which $$ | \lambda_n - n| \leq L < \frac{1}{4}$$ whenever $n\in\Z.$ Then, $ E(\Lambda)=\left\{e^{2\pi i   {\lambda}_n  x }\right\}_{ {\lambda}\in\Lambda}$ is an exponential basis for $L^2(0,1)$ with frame constants $A=\cos(\pi L)-\sin( \pi L)$ and $B= 2-\cos(\pi L)+\sin( \pi L)$.   The constant $\frac{1}{4}$ cannot be replaced by any larger constant. 
\end{Thm}

The theorem is proved using Paley--Wiener theorem and a clever Fourier series expansion of the function $1- e^{2\pi i \delta  x}$. The  quantity 
\begin{equation}\label{def-DL}
	D(L)= 1-\cos(\pi L)+\sin(\pi L)
\end{equation}
plays an important part in the proof of the theorem, as well as in other generalizations.  From the proof of the theorem follows that the frame constants of $E(\Lambda)$ are 
\begin{equation*}
    \begin{split}
      A&= 1-D(L)=\cos(\pi L)-\sin( \pi L),\\  
      B&=1+D(L)=  2-\cos(\pi L)+\sin( \pi L)
    \end{split}
\end{equation*}

Kadec's theorem has been generalized to prove the stability of general exponential frames. See Theorem 1 in \cite{B}.

An important generalization of Kadec's theorem is  due to Avdonin \cite{A}.
  
\begin{Thm}[Special version of the Avdonin's Theorem]\label{T-Avdonin}
Let $\lambda_n=n+\delta_n$ and suppose $\{\lambda_n\}_{n\in\Z}$ is separated, i.e., $\inf_{n\neq k}|\lambda_n-\lambda_k|>0.$ If there exist a positive integer $N$ and a positive real number $\varepsilon<\frac14$ such that
\begin{equation}\label{eqAvdonin}
    \left| \sum_{n=mN+1}^{(m+1)N} \delta_n\right|\leq \varepsilon N
\end{equation}
for all integers $m,$ then the system $\{e^{2\pi i x \lambda_n}\}_{n\in\Z}$ is a Riesz basis for $L^2([0,1]).$
\end{Thm}
This special version of Avdonin's theorem can be found in \cite{S}. The condition $\eqref{eqAvdonin}$ is hard to prove in the case when the sequence is not periodic. But in our case the following Lemma will help.

\begin{Lemma}\label{L-Weyl}
When $g$ is Riemann integrable in $[0,1],$ and periodic of period $1,$ and $\beta$ is irrational, then 
\begin{equation*}
    \lim_{N\rightarrow \infty}\frac{1}{N}\sum_{n=1}^{N} g(n\beta) = \int_0^1 g(x) dx. 
\end{equation*}
\end{Lemma}
This lemma is Corollary 2.3 on page 110 in \cite{SS}. 

\subsection{Bases on disconnected domains}

Let  $E(\Lambda)$  be an exponential basis  on a domain $D\subset  \R^d$. If $D$ is partitioned into disjoint sets $D_1$, ... $D_m,$..., which are  then translated  with translations $\tau_1,\, ... \tau_m,...$ in such a way that the translated pieces do not intersect, then  in general $E(\Lambda)$ is not  a basis on  the "broken domain" $\tilde D= D_1+\tau_1\cup...\cup D_m+\tau_m\cup$...    The following Lemma shows how a  basis $E(\Lambda)$  for $L^2(D)$ can be transformed into  a basis for $L^2(\tilde D)$.

\begin{Lemma}	\label{l-Transform}
Let $D\subset \R^d$ be measurable, with $|D|<\infty$; Let for $m\in\N\cup\{\infty\},$ we have  $ D=\bigcup_{j=0}^{m} D_j$, with $|D_j|>0$ for all $j$ and $k$  $D_k\cap D_j=\emptyset$ when $k\ne j$.  Let $\tau_j$ be translations  such that 
	$ \tau_j(D_j) \cap  \tau_k(D_k)  =\emptyset$  when $k\ne j$. Let $\tilde D = \bigcup_{j=1}^m  \tau_j(D_j)   $. If 
	$\B=\{\psi_n(x)\}_{n\in\N} \subset L^2(D)$ is a Riesz basis for $L^2(D)$, then 
	$$\tilde \B =\Big\{\sum_{j=1}^m \chi_{ \tau_j(D_j)}\psi_n(\tau_j^{-1}x)\Big\}_{n\in\N} $$
	is a Riesz basis for $L^2(\tilde D)$ with the same frame constants.\end{Lemma}

\begin{proof}
	With some abuse of notation, we will let $\tau_j(x)=x+\tau_j $ and $\tau_j(D_j)=D'_j$.  Thus, $\tilde D=\bigcup_{j=1}^m D'_j$. 
	
	Define the operator $T: L^2(D)\rightarrow L^2(\tilde D)$ by $T^{-1}(f)(x)=\sum_{k=1}^{m} f(x-\tau_k )\chi_{D_k}.$ This is a linear transformation. We can also check that  $T$ is invertible, and its inverse is the operator $T^{-1}: L^2(\tilde D)\rightarrow L^2 ( D)$ defined as $T^{-1}(f)(x)=\sum_{k=1}^{m} f(x+\tau_k )\chi_{D'_k}.$ Let us show that  $T$ (and so also  $T^{-1}$) are  isometry.  Indeed, for every $f\in L^2(D)$
	\begin{equation*}
		\begin{split}
			\norm{T(f)}_{L^2(\tilde D)}^2&=\sum_{k=1}^m\norm{T(f)}_{ L^2(D_k') }^2 =\sum_{k=1}^m \int_{\tau_k+D_k } |f(x-\tau_k)|^2dx\\
			&  =\sum_{k=1}^m \int_{D_k } |f(x)|^2dx= ||f||_{L^2(D)}^2,
		\end{split}
	\end{equation*}
	where the third equality comes from   a change of variables in the integrals. An invertible isometry maps  bases into bases and  the frame constants are the same.
	Since $\tilde \B= T(\B)$, we have  proved that $\tilde \B$ is a basis for $L^2(\tilde D).$
\end{proof}

\begin{Rem}
 Let $\{\lambda_n\}_{n\in\Z}\subset\R^d,$ $\{b_k\}_{k=0}^{m-1}\subset\R^d,$ with $m\in\N\cup\{\infty\}$, and $D$ and $\tilde D,$ as in Lemma $\ref{l-Transform}.$ Also, let $ w_n= \sum_{k=0}^{m-1} e^{2\pi i b_k \lambda_n}\chi_{\tilde D_k}.$ The set $ \{e^{2\pi i x \lambda_n}\}_{n\in\Z}$ is a Riesz basis for $L^2(\tilde D) $ if and only if the set $ \{  w_ne^{2\pi i x \lambda_n}\}_{n\in\Z}$ is a Riesz basis for $L^2(D) $. Moreover, those two bases have the same Riesz constants.  
We can also see that the set $ \{ \overline{ w_n}e^{2\pi i x \lambda_n}\}_{n\in\Z}$ is a Riesz basis for $L^2(D) $ if and only if the set 
 	$ \{  e^{2\pi i x \lambda_n}\}_{n\in\Z}$ is a Riesz basis for $L^2(\tilde D) $. 
 	Moreover, those two bases have the same Riesz constants.   

\end{Rem}

If we replace $D$ by $J,$ when 
\begin{equation*} J=\bigcup J_j, \quad\mbox{with }\ 
 J_j=[a_{j}+b_j,a_{j+1}+b_j).\end{equation*} 
 with $J_j=[a_{j}+b_j,a_{j+1}+b_j),$ as in $\eqref{defJ},$ then we obtain a special case of Lemma $\ref{l-Transform}.$ 
 
\begin{Lemma}\label{lmBasis}
For $m$ finite or infinite the sequence $\{g_n\}_{n\in\Z},$ where  
$$g_n=\sum_{k=0}^{m} e^{2\pi i x \lambda_n}e^{2\pi i b_k \lambda_n}\chi_{J_k},$$
is a Riesz basis for $L^2(J)$ if and only if $\B=\{e^{2\pi i x \lambda_n}\}_{n\in\Z}$ is a Riesz basis for $L^2(I).$ Moreover, two bases $\B$ and $\{g_n\}$ have  the same Riesz constants.  Conversely, the set $\{\overset{\sim}{g}_n\}_{n\in\Z},$  where 
$$\overset{\sim}{g}_n=\sum_{k=0}^{m} e^{2\pi i x \lambda_n}e^{-2\pi i b_k \lambda_n}\chi_{I_k},$$
is a Riesz basis for $L^2(I)$ if and only if $\B$ is a Riesz basis for $L^2(J).$ Moreover, two bases $\B$ and $\{\tilde g_n\}$ have  the same Riesz constants.
\end{Lemma}

A version of Lemma \ref{lmBasis} is also in \cite{DK}.

\section{Proofs of  the main results}
In this section, we will prove our main results. But first, we remind the reader that $\B^* =\{e^{2\pi i x(n+\delta_n^*)}\}_{n\in\Z},$ where 
$$\delta_n^*= \begin{cases}
	\frac{\lfloor \beta n \rfloor+1}{\beta}-n  & \text{if } \{\beta n\} \geq \frac12 \\
	\frac{\lfloor \beta n \rfloor}{\beta}-n  & \text{if } \{\beta n\} < \frac12\\
\end{cases},
$$
for some $\beta>0,$ see $\eqref{def-*}$ and $\eqref{def-*d}.$
\subsection{A useful Lemma} 
\begin{Lemma}\label{lmDelta*}
Let $\beta \ge 1.$ Then $\B^*$ is an exponential basis for $L^2([0,1]).$  
\end{Lemma}

\begin{proof}
First, we let $\beta>2.$ Then  
\begin{equation*}\label{eqDeltaStar}
    \sup_{\beta\in\Z}|\delta_n^*|= \sup_{\beta\in\Z}\left|\frac{\dist{(\Z,\beta n)}}{\beta}\right|\leq\frac{1}{2\beta}< \frac{1}{4}.
\end{equation*}
So, by Theorem $\ref{T-Kodec}$ $\left\{e^{2 \pi i x (n + \delta_n^*)}\right\}_{n\in\Z}$ is an exponential basis for $L^2([0,1]).$

Next, let $1\leq \beta \leq 2$ and $\beta\in\Q.$ First, trivial case with $\beta=1$ or $\beta=2.$ In this case $\delta^*_n=0$ for all $n.$ So, $\B^*=\left\{e^{2 \pi i x n}\right\}_{n\in\Z} $,    the standard basis for $L^2([0,1]).$ Now, let $1< \beta < 2$ and $\beta\in\Q,$ so there are two integers $p$ and $q$ such that $\beta=\frac{p}{q}.$  We are going to use Theorem $\ref{T-Avdonin},$ so we need to check if $\{\lambda_n\}_{n\in\Z}=\{n+\delta^*_n\}_{n\in\Z}$ is  separated and  if there exist a positive integer $N$ and a positive real number $\varepsilon<\frac14$ such that
\begin{equation*}
    \left| \sum_{n=mN+1}^{(m+1)N} \delta_n^*\right|\leq \varepsilon N
\end{equation*}
for all integers $m,$ see $\eqref{eqAvdonin}.$ For all $n\in\Z$ we compare $\lambda_n$ and $\lambda_{n+1}$
\begin{equation*}
    \begin{split}
    \lambda_{n}&=n+\begin{cases}
  \frac{\lfloor \beta n \rfloor+1}{\beta}-n  & \text{if } \{ \beta n \} \geq \frac12 \\
  \frac{\lfloor \beta n \rfloor}{\beta}-n  & \text{if }  \{\beta n\}  < \frac12\\
\end{cases}\leq n+\frac{1}{2b},\\
    \lambda_{n+1}&=n+1+\begin{cases}
  \frac{\lfloor \beta (n+1) \rfloor+1}{\beta}-n-1  & \text{if } \{ \beta (n+1) \} \geq \frac12 \\
  \frac{\lfloor \beta (n+1) \rfloor}{\beta}-n-1  & \text{if }  \{\beta(n+1)\}  < \frac12\\
\end{cases}\geq n+1-\frac{1}{2\beta}>\lambda_n.\\
    \end{split}
\end{equation*}
So, the sequence  $\{\lambda_n\}_{n\in\Z}$ is increasing. Moreover,
\begin{equation*}
    \begin{split}
\sup_{n\in\Z}|\lambda_{n+1}-\lambda_n|&\geq n+1-\frac{1}{2\beta} - \left(n+\frac{1}{2\beta}\right)= 1- \frac1\beta>0.
    \end{split}
\end{equation*}
Therefore, $\{\lambda_n\}_{n\in\Z}$ is separated. 

Next, we observe that for all $n\in\Z$
\begin{equation*}
\begin{split}
    \delta^*_{n+q} &=
\begin{cases}
  \frac{\lfloor \frac{p}{q} (n+q) \rfloor+1}{\frac{p}{q}}- n -q = \frac{\lfloor \frac{p}{q} n \rfloor + p+1}{\frac{p}{q}}-n-q  & \text{if } \left\{ \frac{p}{q} (n-q)  \right\} \geq \frac12 \\
  \frac{\lfloor \frac{p}{q} (n+q) \rfloor}{\frac{p}{q}}- n -q = \frac{\lfloor \frac{p}{q} n \rfloor + p}{\frac{p}{q}}-n-q  & \text{if }   \left\{ \frac{p}{q} (n-q)  \right\} < \frac12\\
\end{cases}\\
&= \begin{cases}
\frac{\lfloor \frac{p}{q} n \rfloor+1}{\frac{p}{q}}- n  & \text{if }   \left\{ \frac{p}{q} n  \right\} \geq \frac12\\
    \frac{\lfloor \frac{p}{q} n \rfloor}{\frac{p}{q}}- n  & \text{if }  \left\{ \frac{p}{q} n  \right\} < \frac12\\
\end{cases}.
\end{split}
\end{equation*}
So, $\delta^*_n=\delta^*_{n+q}$ and we also observe that  $\delta^*_q=0.$ Thus, in order to apply $\eqref{eqAvdonin}$ it is enough to consider $\left|\sum_{n=1}^{q-1} \delta^*_n \right|.$ Now, for all $n=1,...,\lfloor\frac{q}{2}\rfloor$ 
\begin{equation*}
\begin{split}
    \delta^*_{q-n} &=
\begin{cases}
   \frac{\lfloor \frac{p}{q} (q-n) \rfloor+1}{\frac{p}{q}}- q +n = \frac{\lfloor -\frac{p}{q} n \rfloor + p + 1}{\frac{p}{q}}-q+n  & \text{if } \left\{ \frac{p}{q} (q-n)  \right\} \geq \frac12 \\
\frac{\lfloor \frac{p}{q} (q-n) \rfloor}{\frac{p}{q}}- q +n = \frac{\lfloor -\frac{p}{q} n \rfloor + p}{\frac{p}{q}}-q+n & \text{if } \left\{ \frac{p}{q} (q-n)  \right\} < \frac12\\
\end{cases}\\
&= \begin{cases}
-\frac{\lfloor \frac{p}{q} n \rfloor+1}{\frac{p}{q}}+ n & \text{if }   \left\{ \frac{p}{q} n  \right\} \geq \frac12\\
    -\frac{\lfloor \frac{p}{q} n \rfloor}{\frac{p}{q}}+ n  & \text{if }  \left\{ \frac{p}{q} n  \right\} < \frac12\\
\end{cases}.
\end{split}
\end{equation*}
It means that $\delta^*_n + \delta^*_{q-n}=0.$ Moreover, if $q$ is even, then $\delta^*_{\frac{q}{2}}+\delta^*_{\frac{q}{2}}=0,$ then $\delta^*_{\frac{q}{2}}=0.$ Thus, 
$$
\left|\sum_{n=1}^{q-1} \delta^*_n \right| = 0
$$
If $q$ is odd, then 
\begin{equation*}
    \begin{split}
     \left|\sum_{n=1}^{q} \delta^*_n \right| &= |\delta^*_{\frac{q+1}{2}}| =\left|\begin{cases}
  \frac{\lfloor \frac{p}{2} \rfloor+1}{\frac{p}{q}}-\frac{q}{2}  & \text{if } \left\{ \frac{p}{2} \right\} \geq \frac12 \\
  \frac{\lfloor \frac{p}{2} \rfloor}{\frac{p}{q}}-\frac{q}{2}   & \text{if }  \left\{ \frac{p}{2} \right\} < \frac12\\
\end{cases}  \right|= \begin{cases}
  0  & \text{if } p \text{ is even} \\
  \frac{1}{2b}   & \text{if } p \text{ is odd}\\
\end{cases}< \frac{1}{4}<\frac{q}{4},
    \end{split}
\end{equation*}
because $q>1.$ Therefore, using Theorem $\ref{T-Avdonin}$ we conclude that $\left\{e^{2 \pi x (n + \delta_n^*)}\right\}_{n\in\Z}$ is an exponential basis for $L^2([0,1]).$

Next, we consider the case when $1<\beta<2$ is irrational. We can rewrite our sequence in the form 
$$
\delta_n^* = \frac{g(n \beta)}{\beta},
$$
where 
$$
g(x)=\begin{cases}
  1 - \{ x \}  & \text{if } \{ x \} \geq \frac12 \\
  - \{ x \} & \text{if }  \{ x \} < \frac12\\
\end{cases}.
$$
$g$ is Riemann integrable in $[0,1],$ and periodic of period $1.$ Moreover, 
$$
\int_0^1 g(x)dx = 0.
$$
So, by Lemma $\ref{L-Weyl}$
$$
\lim_{N\rightarrow \infty}\frac{1}{N}\sum_{n=1}^{N} g(n\beta)=0
$$
Moreover, by simple translation we can get that for all $m\in \Z$ 
$$
\lim_{N\rightarrow \infty}\frac{1}{N}\sum_{n=mN+1}^{(m+1)N} \frac{g(n\beta)}{\beta}=0.
$$
It means that for any $\varepsilon<\frac{1}{4},$ there is a $N_0$ such that for all $m\in\Z$ and for all $N>N_0$
$$
\left|\frac{1}{N}\sum_{n=mN+1}^{(m+1)N} \frac{g(n\beta)}{\beta}\right| =\left|\frac{1}{N}\sum_{n=mN+1}^{(m+1)N} \delta_n^*\right| <\varepsilon.
$$
Therefore, using Theorem $\ref{T-Avdonin}$ we conclude that $\left\{e^{2 \pi x (n + \delta_n^*)}\right\}_{n\in\Z}$ is an exponential basis for $L^2([0,1]).$
\end{proof}
\begin{Rem}
 Since we proved that $\delta_n^*$ is periodic  when $1\leq \beta \leq 2$ and $\beta\in\Q,$ we could also have used Corollary 3.1 from \cite{D} to conclude that $\B^*$ is an exponential basis for $L^2([0,1]).$
\end{Rem}

\begin{Rem}\label{r2}
If $\beta>2$ then Theorem \ref{T-Kodec}  shows  that the frame constant of the basis are $A=\cos(\pi L)-\sin( \pi L)$ and $B= 2-\cos(\pi L)+\sin( \pi L)$, where $L=\sup_{n\in\Z}|\delta_n|.$
\end{Rem}

\subsection{Proof of Theorem $\ref{l-segments}$}
\begin{proof}
Let $m$ be infinite or finite. By Lemma $\ref{lmDelta*}$  $\B^*$ is an exponential basis for $L^2(I)$ with the Riesz constants $A$ and $B.$ We can obtain a Riesz basis $\{g_n\}_{n\in\Z}$ for $L^2(J)$ using Lemma $\ref{lmBasis},$ where $g_n=\sum_{k=0}^{m-1} e^{2\pi i (x-b_k) (n+\delta_n^*)}\chi_{J_k}.$ Next, we use Paley-Wiener Theorem to show that $\B^*$ is a basis for $L^2(J)$. So, we need to show that there is $0\leq\alpha<1$ such that for all sequence $\{a_n\}$ with the property $\sum |a_n|^2= 1$  
\begin{equation}\label{eqPWJ-l}
    ||\sum a_n (g_n-e^{2\pi ix (n+\delta_n^*)})  ||^2_{L^2(J)}\leq \alpha ||\sum a_n g_n  ||^2_{L^2(J)}.
\end{equation}
Using a simple substitution and the Riesz constants of the basis $\B^*$ we can estimate the right-hand side of the inequality $\ref{eqPWJ-l}$ 
\begin{equation*}
    0<A\leq ||\sum a_n e^{2\pi i x (n+\delta_n^*)}  ||^2_{L^2(I)}=||\sum a_n g_n  ||^2_{L^2(J)}.
\end{equation*}
For the left-hand side, using the definition of $g_n$ and Minkowski's inequality 
\begin{equation*}
\begin{split}
    ||\sum a_n (g_n-e^{2\pi ix (n+\delta_n^*)})  ||^2_{L^2(J)} &\leq \sum_{k=1}^{m-1} ||\sum a_n (e^{2\pi i(x-b_k) (n+\delta_n^*)}-e^{2\pi ix (n+\delta_n^*)})  ||^2_{L^2(J_k)}\\
    &= \sum_{k=1}^{m-1} ||\sum  a_n (e^{-2\pi ib_k (n+\delta_n^*)}-1) e^{2\pi ix (n+\delta_n^*)}  ||^2_{L^2(J_k)}.
\end{split}
\end{equation*}
Next, we recall that 	\begin{equation*}
		\delta_n^*=\begin{cases}
	\frac{\lfloor \beta n \rfloor+1}{\beta}-n  & \text{if } \{\beta n\} \geq \frac12 \\
	\frac{\lfloor \beta n \rfloor}{\beta}-n  & \text{if } \{\beta n\} < \frac12\\
\end{cases}.
	\end{equation*} 
 It means that for each $n\in \Z$ we can find $M_{n}\in \Z$ such that $\delta_n^*=\frac{M_{n}}{\beta}-n.$ Thus, for all $k=1,...,m-1.$
\begin{equation*}
\begin{split}
 e^{-2\pi  ib_k (n+\delta_n^*)}-1&= e^{-2\pi i \frac{M_{n}b_k}{\beta}}-1=0,
\end{split}
\end{equation*}
because $\frac{b_k}{\beta}\in\Z$ for $k=1,...,m-1.$ So, 
\begin{equation*}
\begin{split}
    ||\sum a_n (g_n-e^{2\pi ix (n+\delta_n^*)})  ||^2_{L^2(J)} =0.
\end{split}
\end{equation*}
and the inequality $\eqref{eqPWJ-l}$ holds. Therefore, $\B^*$ is an exponential basis for $L^2(J)$ with Riesz constants $A$ and $B.$

\end{proof}

 \subsection{Proof of Theorem $\ref{T-ComplementB}$}
In \cite{PRW} the reader can find the following result.
\begin{Thm}\label{T-Complement}
Let $\Delta>0$ and $S\subset [0,\Delta].$ Suppose that for some $\Lambda\subset \frac{1}{\Delta}\Z,$ $\E(\Lambda)$ is a Riesz basis for $L^2(S).$ Then $\E(\frac{1}{\Delta}\Z\setminus\Lambda)$ is a Riesz basis for $L^2([0,\Delta]\setminus S).$ 
\end{Thm}

\begin{proof} [Proof of the Theorem \ref{T-ComplementB}]
Let $\beta$ and $b_k$ be real numbers as in Theorem \ref{T-ComplementB}. First, we introduce the interval $[0,\Delta],$ where $\Delta=\lceil\frac{1+b_{m-1}}{\beta} \rceil \beta$ and $\lceil x \rceil$  is a ceiling function. The set $J$ defined as $\eqref{defJ}$ will be a subset of $[0,\Delta].$ Let  $\Lambda=\{n+\delta^*_n\}_{n\in\Z}=\{\lambda_n\}_{n\in\Z},$ where 
\begin{equation*}
    \begin{split}
        \lambda_n&=\begin{cases}
	\frac{\lfloor \beta n \rfloor+1}{\beta}  & \text{if } \{\beta n\} \geq \frac12 \\
	\frac{\lfloor \beta n \rfloor}{\beta}   & \text{if } \{\beta n\} < \frac12\\
\end{cases}.
    \end{split}
\end{equation*}
In view of the definition of $\beta$ we have that $\Lambda\subset\frac{1}{\Delta}\Z.$ By Theorem \ref{l-segments} $\E(\Lambda)$ is a Riesz basis for $L^2(J).$ So, by Theorem $\ref{T-Complement},$ $\E(\frac{1}{\Delta}\Z\setminus\Lambda)$ is a Riesz basis for $L^2([0,\Delta]\setminus J).$
 \end{proof}

\section{Extensions and generalizations}
\subsection{A stability Theorem}
\begin{Thm}\label{T-stability} (Stability Theorem for $m$-segments). \label{l-segmentsStability}
Let $m\in\N\setminus\{1\}$ and $b_j \in \N$ for all $j=1,...,m-1$ and $L:=\sup_n |\delta_n|<\frac{1}{4}.$ Then the system $\{e^{2\pi i(n+\delta_n)x}\}$ is a Riesz basis for $L^2(J)$ if for all $j=1,...,m-1$ and $n\in \Z$
	\begin{equation}\label{eqStability-l} 
    		\max_{j=1,..,m-1} \{B_{\gamma_j} \sup_n |\sin(\pi \dist(\Z, \delta_nb_j)  ) |\}   < \frac{ A}{2\sqrt{m}},
	\end{equation}
 where 
 \begin{equation}\label{eqALBL}
     \begin{split}
        A&=A(L)=\cos(\pi L) - \sin(\pi L);\\ 
        B_{\gamma_j}&=B_{\gamma_j}(L)=2-\cos{(\pi \gamma_j L)}+\sin{(\pi \gamma_j L)}.
     \end{split}
 \end{equation}
 \end{Thm}
 
\begin{proof}
By $\frac{1}{4}-$Kadec Theorem if $L:=\sup_n |\delta_n|< \frac{1}{4},$ then $e^{2\pi i x (n+\delta_n)}$ is the Riesz basis for $L^2(I)$ with constants $A=\cos{(\pi L)} - \sin{(\pi L)}$ and  $B=2-\cos{(\pi L)} + \sin{(\pi L)}.$ Then using Lemma $\ref{lmBasis}$ we can obtain a Riesz basis $\{g_n\}_{n\in \Z}$ for $L^2(J),$ where
$$g_n=\sum_{k=0}^{m} e^{2\pi i x \lambda_n}e^{2\pi i b_k \lambda_n}\chi_{J_k}.$$
Next, we are going to use Theorem $\ref{T-PaleyWiener}$  with the norm $$||\cdot||^2_{L^{2,\infty}(J)}=\sup_{k=0,...,m-1}||\cdot||^2_{L^2(J_k)}.$$ 
The right-hand side of the inequality we can estimate using substitution and the Riesz sequence definition as 
\begin{equation*}
    \begin{split}
        ||\sum a_n g_n||_{L^{2,\infty}(J)}&\geq \frac{1}{\sqrt{m}}||\sum a_n g_n||_{L^2(J)}\\
        &=\frac{1}{\sqrt{m}}||\sum a_n e^{i(n+\delta_n)x}||_{L^2(I)}.
    \end{split}
\end{equation*}
By the elementary inequality $\sum\{|a_1|,...,|a_m|\}\geq \frac{1}{\sqrt(m)}\sqrt{\sum_n|a_n|^2}$ we have
\begin{equation*}
    \begin{split}
\frac{1}{\sqrt{m}}||\sum a_n e^{i(n+\delta_n)x}||_{L^2(I)}
        \geq \frac{A}{\sqrt{m}} \sqrt{\sum |a_n|^2}=\frac{A}{\sqrt{m}}.
    \end{split}
\end{equation*}
For the left-hand side, we have
\begin{equation*}
    \begin{split}
        ||\sum a_n \left(g_n- e^{2 \pi i(n+\delta_n)x}\right)||_{L^{2,\infty}(J)}&\leq \max_{j=1,..,m}||\sum a_n e^{2\pi i(n+\delta_n)x}\left(e^{-2\pi i(n+\delta_n)b_j}- 1\right)||_{L^{2,\infty}(J_j)}\\
        &\leq \max_{j=1,..,m}\{ B_{\gamma_j}\sqrt{\sum \left|a_n \left(e^{-2\pi i(n+\delta_n)b_j}- 1\right)\right|^2}\}\\
        &\leq 2 \max_{j=1,..,m} \{B_j \sup_n{|\sin{(\pi(n+\delta_n)b_j)}|}\}\\
        &= 2 \max_{j=1,..,m} \{B_{\gamma_j} \sup_n{|\sin{(\pi\delta_nb_j)}|}\},
    \end{split}
\end{equation*}
where $B_{\gamma_j}= 2 - \cos{(\pi\gamma_j L)} + \sin{(\pi \gamma_j L)}$ for $j=1,..,m-1.$ So, we need 
$$
 \max_{j=1,..,m-1} \{B_{\gamma_j} \sup_n{|\sin{(\pi\delta_nb_j)}|}\} < \frac{ A}{2\sqrt{m}}
$$
or
$$
 \max_{j=1,..,m-1} \{B_{\gamma_j} \sup_n{|\sin{(\min\{\{\delta_nb_j\},1-\{\delta_nb_j\}\}\pi)}|}\} < \frac{ A}{2\sqrt{m}}.
$$

\end{proof}

\begin{Rem}
    If $m=1$ we only  have one interval, so Kadec's theorem holds.  
\end{Rem}

\begin{Rem}
Observe that $A(L)$ defined as $\eqref{eqALBL}$ is a concave down function of $L$ on the interval $\left[0,\frac14\right)$ and  $B_{\gamma_j}(L)$ defined as $\eqref{eqALBL}$ is a concave up function of $L$ on the interval $\left[0,\frac{1}{4\gamma_j}\right)$ for $j=1,...,m-1.$ Also, we can use the fact that $\sin(\pi \mbox{dist}(\delta_nb_j, \Z)) \leq  \pi \mbox{dist}(\delta_nb_j, \Z). $ So, 
the condition 
$$
\max_{j=1,..,m-1} \{(1+4L\gamma_j)\,\mbox{dist}(\Z, \delta_nb_j)\}    \leq \frac{ (1-4 L)}{2 \sqrt{m}\pi}, 
$$    
for $j=1,...,m-1,$ guarantees that $\eqref{eqStability-l}$ holds.  
\end{Rem}

\subsection{ $d-$dimensions}

 We use the arguments developed in previous sections to find bases on "split cubes". Let $Q=[0,1)^d$ and let $0=a_{0,k}<a_{1,k}<...<a_{m_k-1,k}<...<1$, where $m_k\in\N\cup\{\infty\},$ for all $k=1,...,d.$ Also, we let $I_{j,k}=[a_{j,k}, a_{j+1,k})$ and
$$R_{\vec{j}}=R_{j_1,j_2,...,j_d}=I_{j_1,1}\times I_{j_2,2}\times...\times I_{j_k,k},$$ 
so 
$$Q=  \bigcup_{j_1,j_2,...,j_d} R_{j_1,j_2,...,j_d}.$$  
	
	Given that for all $k=1,...,d,$ we have  $0=b_{0,k}<b_{1,k}<...<b_{m_k-1,k}$, and we let  $\vec{\beta}_{\vec{j}}=(b_{j_1,1},b_{j_2,2},...,b_{j_d,d}).$ Now, we can define a "split cubes" as 
\begin{equation}\label{defQ} 
\tilde Q=\bigcup_{\vec{j}} \left(\tau_{\vec{\beta}_{\vec{j}}}+R_{\vec{j}}\right).
 \end{equation}  
Let $\beta_k>0$ be fixed for $k=1,...,d$;  for every $n_k\in \Z$, we let
	\begin{equation} \label{d*d}
		\delta_{n_k}^*=\begin{cases}
	\frac{\lfloor \beta_k n_k \rfloor+1}{\beta_k}-n_k  & \text{if } \{\beta_k n_k\} \geq \frac12 \\
	\frac{\lfloor \beta_k n_k \rfloor}{\beta_k}-n_k  & \text{if } \{\beta_k n_k\} < \frac12\\
\end{cases}.
	\end{equation}
	We let $\vec{\delta}_{n}^*=(\delta_{n_1}^*,...,\delta_{n_d}^*).$ Let 
\begin{equation}\label{def-*d}
\begin{split}
    \B_k^* &=\{e^{2\pi i x_k(n_k+\delta_{n_k}^*)}\}_{n_k\in\Z },\\ 
    \B^*&= \{e^{2\pi i  x\cdot (n +\vec\delta_{n }^*) }\}_{n \in\Z^d} \quad \mbox{ with $x\in\R^d$.}
\end{split}
	 \end{equation} 
 Lemma 2.1 in \cite{SZ} can be generalized in the following way. 
 
  \begin{Lemma}\label{l-Tensor}
Let each the set  ${\cal U}_j=\{e^{2\pi i   x  \lambda_{j }(n)}\}_{n \in\Z}$ be a  basis on a domain $D_j\subset \R$, with constants $A_j$ and $B_j,$  then the set 
 $\{e^{2\pi i (   \lambda_1(n_1)x_1 +... + \lambda_d(n_d) x_d }\}_{n_1, ..., n_d \in\Z } $ is a basis on $L^2( D_1\times... \times D_d),
 $ with the constants $A=A_1\cdot...\cdot A_d$ and $B=B_1\cdot...\cdot B_d.$
\end{Lemma}
 \begin{proof}
We consider only the case $d=2.$ If $d>2,$ the proof is similar. Let  ${\cal U}_j=\{e^{2\pi i   x  (\lambda_{j }(n)}\}_{n \in\Z}$ be a basis on a domain $D_j\subset \R$, with constants $A_j$ and $B_j,$ $j=1,2.$ Also, to simplify formulas we use the following notations
\begin{equation*}
    \begin{split}
        v_{j,n_j}= e^{2\pi i   x  \lambda_{j }(n_j)} \text{ and } \tilde v_{n_1,n_2}=v_{1,n_1}\cdot v_{2,n_2}.  
    \end{split}
\end{equation*}
For any $f\in L^2(D_1\times D_2),$ we have for  
\begin{equation*}
    \begin{split}
        \sum_{n_1, n_2\in \Z} |\l  f, \tilde v_{n_1,n_2}\r_{L^2(D_1\times D_2)} |^2 &= \sum_{n_2} \sum_{n_1} \left|\int_{D_1} \left(\int_{D_2} f(x_1,x_2)v_{2,n_2} dx_2\right) v_{1,n_1} dx_1\right|^2\\
        &\leq B_1 \int_{D_1} \sum_{n_2} \left| \int_{D_2} f(x_1,x_2)v_{2,n_2} dx_2\right|^2 dx_1\\
        &\leq B_1 B_2 \norm{f}_{L^2(D_1\times D_2)}^2. 
    \end{split}
\end{equation*}
A similar argument shows that 
\begin{equation*}
    \begin{split}
        \sum_{n_1, n_2\in \Z} |\l  f, \tilde v_{n_1,n_2}\r_{L^2(D_1\times D_2)} |^2 
        &\geq A_1 A_2 \norm{f}_{L^2(D_1\times D_2)}^2.
    \end{split}
\end{equation*}
Thus,  $\{e^{2\pi i (\lambda_1(n_1)x_1 + \lambda_2(n_2)x_2 }\}_{n_1, n_2 \in\Z } $ is a frame for $L^2(D_1\times D_2).$ Next, for any finite sequence of complex numbers  $c_{n_1,n_2 }$ we have 
\begin{equation*}
\begin{split}
        \Big\Vert \sum_{n_1}\sum_{n_2}  c_{n_1,n_2 }  \tilde v_{n_1,n_2} \Big\Vert^2 &=\int_{D_2} \int_{D_1} \left|\sum_{n_1} \left(\sum_{n_2} c_{n_1,n_2 } v_{2,n_2}  \right) v_{1,n_1}  \right|^2 dx_1dx_2 \\
        &\leq B_1 \sum_{n_1} \int_{D_2} \left| \sum_{n_2} c_{n_1,n_2 } v_{2,n_2}  \right|^2 dx_2\\
        &\leq B_1 B_2 \sum_{n_1} \sum_{n_2}  \left|  c_{n_1,n_2 }  \right|^2 
\end{split}
\end{equation*}
A similar argument shows that 
\begin{equation*}
\begin{split}
        \Big\Vert \sum_{n_1}\sum_{n_2}  c_{n_1,n_2 }  \tilde v_{n_1,n_2} \Big\Vert^2 &\geq A_1 A_2 \sum_{n_1} \sum_{n_2}  \left|  c_{n_1,n_2 }  \right|^2 
\end{split}
\end{equation*}
Therefore,  $\{e^{2\pi i (\lambda_1(n_1)x_1 + \lambda_2(n_2)x_2 }\}_{n_1, n_2 \in\Z } $ is a basis for $L^2(D_1\times D_2).$ Moreover, $A=A_1\cdot A_2$ and $B=B_1\cdot B_2$ are the Riesz constants.  
 \end{proof}
 
 Now, we can use Lemma \ref{l-Tensor} to generalize some results from section 3 in $d-$dimensions.  
 
 \begin{Lemma}\label{lmDelta*d}
For all $k=1,...,d,$ let $\beta_k\geq 1.$ Then $\B^*$ is an exponential basis for $L^2([0,1]^d),$ where $\B^*$ is defined as in $\eqref{def-*d}.$  
\end{Lemma}
\begin{proof}
From Lemma $\ref{lmDelta*}$ we have that $\B^*_k,$ defined as in $\eqref{def-*d},$ is a basis for $L^2([0,1]).$ Therefore, by Lemma $\ref{lmDelta*d}$ $\B^*$ is an exponential basis for $L^2([0,1]^d).$
\end{proof}

 \begin{Thm}\label{l-segmentsD}
Let $\vec{m}=(m_1,...,m_d),$ when $m_k\in\N\cup\{\infty\}$. For all $k=1,...,d,$ let $\beta_k\geq 1.$ If for all $k=1,...,d$  $\frac{b_{j,k}}{\beta_k}\in\Z $ for all $j$, then the set $\B^*$ defined in \eqref{def-*d} is an exponential basis for $L^2(\tilde Q),$ where $\B^*$ is defined as in $\eqref{def-*d}.$ Moreover, $\B^*$ has the same frame constants for $L^2(\tilde Q)$ and $L^2(Q).$
\end{Thm}
\begin{proof}
From Theorem $\ref{l-segments}$ we have that $\B^*_k,$ defined as in $\eqref{def-*d},$ is a basis for $L^2(D_k),$ where $D_k$ is a projection of $\tilde Q$ on $k$th coordinate. Therefore, by Lemma $\ref{lmDelta*d}$ $\B^*$ is an exponential basis for $L^2(\tilde Q).$
\end{proof}

\section{Remarks and open problems}
Theorem \ref{l-segments} provides explicit exponential bases for split intervals under conditions on $b_k.$  In Remark \ref{r2}, we have observed that   we   can obtain  the frame constants for the basis  when $b\geq 2.$ The problem of finding explicit exponential bases for general split intervals, and  explicit frame constants for these bases,   is still waiting for a solution.  The  same situation occurs with exponential bases on split cubes  in $\R^d$.

 We have provided explicit  exponential bases on certain  infinite unions of intervals of  total finite measure.   We would like to generalize our results  and prove the existence of exponential bases on  arbitrary infinite unions of intervals or rectangles. 
	
Our	Theorem \ref{T-stability} reduces to Kadec's theorem  when the interval is not split, but  in the other cases, we obtain stability bounds that depend on the  gaps between the intervals.  We believe that this result can be improved, and we hope to do so in another paper.

\end{document}